\documentclass[a4paper, 12pt, twoside, onecolumn]{article}
\usepackage{amsmath,amssymb,amsthm,mathrsfs,hyperref}
\usepackage[all]{xy}
\usepackage[top=1in, bottom=30mm, left=30mm, right=1in]{geometry}

\DeclareMathOperator{\ann}{ann}
\DeclareMathOperator{\Ass}{Ass}
\DeclareMathOperator{\cx}{cx}
\DeclareMathOperator{\gr}{gr}
\DeclareMathOperator{\Image}{Image}
\DeclareMathOperator{\Ext}{Ext}
\DeclareMathOperator{\Tor}{Tor}
\DeclareMathOperator{\Hom}{Hom}
\DeclareMathOperator{\projdim}{projdim}
\DeclareMathOperator{\Spec}{Spec}
\DeclareMathOperator{\Proj}{Proj} 
\DeclareMathOperator{\Var}{Var}

\theoremstyle{plain}
\newtheorem{theorem}{Theorem}[section]
\newtheorem{lemma}[theorem]{Lemma}

\newtheorem{corollary}[theorem]{Corollary}

\theoremstyle{definition}
\newtheorem{hypothesis}[theorem]{Hypothesis}
\newtheorem{para}[theorem]{}

\theoremstyle{remark}

\numberwithin{equation}{theorem}

\title{Asymptotic prime divisors over complete intersection rings}
\author{Dipankar Ghosh\\Tony J. Puthenpurakal}
\date{27th March, 2014}

\begin{document}

\maketitle

\begin{abstract}
\noindent
 Let $A$ be a local complete intersection ring. Let $M,N$ be two finitely generated $A$-modules and $I$ an ideal of $A$.
 We prove that
 \[
  \bigcup_{i\geqslant 0}\bigcup_{n \geqslant 0}\Ass_A\left(\Ext_A^i(M,N/I^n N)\right)
 \]
 is a finite set. Moreover, we prove that there exist $i_0,n_0\geqslant 0$ such that for all $i\geqslant i_0$ and $n \geqslant n_0$,
 we have
 \begin{align*}
  \Ass_A\left(\Ext_A^{2i}(M,N/I^nN)\right) &= \Ass_A\left(\Ext_A^{2 i_0}(M,N/I^{n_0}N)\right), \\
  \Ass_A\left(\Ext_A^{2i+1}(M,N/I^nN)\right) &= \Ass_A\left(\Ext_A^{2 i_0 + 1}(M,N/I^{n_0}N)\right).
 \end{align*}
 We also prove the analogous results for complete intersection rings which arise in algebraic geometry. Further, we
 prove that the complexity $\cx_A(M,N/I^nN)$ is constant for all sufficiently large $n$.
\end{abstract}


 
\section{Introduction}\label{Introduction}

 Let $A$ be a commutative Noetherian ring, $I$ an ideal of $A$, and $M$ a finitely generated $A$-module.
 Brodmann \cite{Br79} proved that the set of associated prime ideals $\Ass_A(M/I^n M)$ is independent of $n$ for all
 sufficiently large $n$. Thereafter, L. Melkersson and P. Schenzel generalized Brodmann's result in
 \cite[Theorem 1]{Me93} by proving that
 \[\Ass_A\left(\Tor_i^A(M,A/I^n)\right)\]
 is independent of $n$ for all large $n$ and for a fixed $i\geqslant 0$.
 
 Later, D. Katz and E. West proved the above result in a more general way \cite[3.5]{Ka04}; if $N$ is a finitely
 generated $A$-module, then for a fixed $i\geqslant 0$, the sets
 \[
  \Ass_A\left(\Tor_i^A(M,N/I^n N)\right)\quad\mbox{and}\quad \Ass_A\left(\Ext_A^i(M,N/I^n N)\right)
 \]
 are stable for all large $n$. So, in particular, for a fixed $i\geqslant 0$,
 \[
  \bigcup_{n\geqslant 0}\Ass_A\left(\Ext_A^i(M,N/I^n N)\right)
 \]
 is a finite set. In this context, the second author \cite[page 368]{Pu13} raised a question about what happens if we
 vary $i$ ($\geqslant 0$) also? More precisely,
 \begin{center}
  ($\dagger$)\hfill is the set ~$\displaystyle\bigcup_{i\geqslant 0}\bigcup_{n\geqslant 0}\Ass_A\left(\Ext_A^i(M,N/I^n N)\right)$
  finite?\hfill \;
 \end{center}
 The motivation for the question ($\dagger$) came from the following two questions. They were raised by
 W. Vasconcelos \cite[3.5]{Va98} and Melkersson and Schenzel \cite[page 936]{Me93} respectively.
 \begin{center}
  (1)\hfill Is the set ~$\displaystyle\bigcup_{i\geqslant 0}\Ass_A\left(\Ext_A^i(M,A)\right)$ finite?\hfill \; \\
  (2)\hfill Is the set ~$\displaystyle\bigcup_{i\geqslant 0}\bigcup_{n\geqslant 1}\Ass_A\left(\Tor_i^A(M,A/I^n)\right)$ finite?\hfill\;
 \end{center}

 Recently, the second author \cite[Theorem~5.1]{Pu13} proved that if $A$ is a local complete intersection
 ring and $\mathcal{N}=\bigoplus_{n\geqslant 0}N_n$ is a finitely generated graded module over the Rees ring $\mathscr{R}(I)$,
 then
 \[ \bigcup_{i\geqslant 0} \bigcup_{n\geqslant 0} \Ass_A \left( \Ext_A^i(M,N_n) \right) \]
 is a finite set. Moreover, he proved that there exist $i_0, n_0 \geqslant 0$ such that
 \begin{align*}
  \Ass_A\left(\Ext_A^{2i}(M,N_n)\right) &= \Ass_A\left(\Ext_A^{2 i_0}(M,N_{n_0})\right), \\
  \Ass_A\left(\Ext_A^{2i+1}(M,N_n)\right) &= \Ass_A\left(\Ext_A^{2 i_0 + 1}(M,N_{n_0})\right)
 \end{align*}
 for all $i\geqslant i_0$ and $n \geqslant n_0$. In particular, if $N$ is a finitely generated $A$-module, then $\mathcal{N}$ can be
 taken as $\bigoplus_{n\geqslant 0}(I^n N)$ or $\bigoplus_{n\geqslant 0}(I^n N/I^{n+1} N)$.
 In the present study, we prove that the question $(\dagger)$ has an affirmative answer for a local
 complete intersection ring. We also analyze the stability of the sets of associated prime ideals which occurs
 periodically after a certain stage.

 Let $(A,\mathfrak{m})$ be a local complete intersection ring. Let $M,N$ be two finitely generated $A$-modules and
 $I$ an ideal of $A$. The complexity of the pair $(M,N)$ is defined to be the number
 \[
  \cx_A(M,N) = \inf\left\{b\in\mathbb{N} ~\middle|~ \limsup_{n\rightarrow\infty}\dfrac{\mu(\Ext_A^n(M,N))}{n^{b-1}}<\infty \right\},
 \]
 where $\mu(D)$ denotes the minimal number of generators of a finitely generated $A$-module $D$. In \cite[Theorem 7.1]{Pu13}, the
 second author proved that $\cx_A(M,I^j N)$ is constant for all $j \gg 0$. In Section~\ref{Support varieties}, we prove that
 \begin{center}
  $(\dagger\dagger)$\hfill $\cx_A(M,N/I^jN)$\quad is constant for all $j\gg 0$. \hfill \;
 \end{center}

 The rest of the paper is organized as follows. In Section~\ref{Module structure}, we give some graded module structures which we use
 in order to prove our main results. The finiteness results on asymptotic primes are proved in
 Section~\ref{Asymptotic associated primes: Finiteness}; while the stability results are shown in
 Section~\ref{Asymptotic associated primes: Stability}. Then, in Section~\ref{Asymptotic associated primes: The geometric case}, we
 prove the analogous results on associated primes for complete intersection rings which arise in algebraic geometry. Finally, in
 Section~\ref{Support varieties}, we prove $(\dagger\dagger)$.
 
\section{Module structure}\label{Module structure}
 In this section, we give the graded module structures which we are going to use in order to prove our main results.
 
 Let $Q$ be a commutative Noetherian ring and ${\bf f} = f_1,\ldots,f_c$ a $Q$-regular sequence. Set $A := Q/({\bf f})$. Let $M$
 and $D$ be two $A$-modules, and let $M$ be finitely generated.
 
 \begin{para}\label{para:module structure 1}
 Let $\mathbb{F}:\;\cdots\rightarrow F_n\rightarrow\cdots\rightarrow F_1\rightarrow F_0\rightarrow 0$ be a projective resolution of
 $M$ by finitely generated free $A$-modules. Let $t_j:\mathbb{F}(+2)\to\mathbb{F}$, $1\leqslant j\leqslant c$ be the Eisenbud operators defined
 by ${\bf f} = f_1,\ldots,f_c$ (see \cite[Section~1]{Ei80}). By \cite[1.4]{Ei80}, the maps $t_j$ are determined uniquely up to
 homotopy. In particular, they induce well-defined maps
 \[ t_j : \Ext_A^i(M,D) \longrightarrow \Ext_A^{i+2}(M,D) \quad \mbox{for all $i$ and $1 \leqslant j \leqslant c$}. \]
 By \cite[1.5]{Ei80}, the maps $t_j$ ($j=1,\ldots,c$) commute up to homotopy. Thus
 \[\Ext_A^{\star}(M,D) := \bigoplus_{i\geqslant 0}\Ext_A^i(M,D)\]
 turns into a graded $\mathscr{T}:=A[t_1,\ldots,t_c]$-module, where $\mathscr{T}$ is the graded polynomial ring  over 
 $A$ in the cohomology operators $t_j$ defined by ${\bf f}$ with $\deg(t_j)=2$ for all $1\leqslant j\leqslant c$. Furthermore,
 these structures depend only on ${\bf f}$, are natural in both module arguments and commute with the connecting
 maps induced by short exact sequences.
 \end{para}
 
 \begin{para}\label{para:Gulliksen}
 If $\projdim_Q(M)$ is finite and $D$ is
 a finitely generated $A$-module, then Gulliksen \cite[3.1]{Gu74} proved that $\Ext_A^{\star}(M,D)$ is a finitely
 generated graded $\mathscr{T}$-module.
 \end{para}
 
 \begin{para}\label{para:module structure 2}
 Let $I$ be an ideal of $A$. Let $\mathscr{R}(I) = \bigoplus_{n \geqslant 0} I^n X^n$ be the Rees ring associated to $I$. We consider
 $\mathscr{R}(I)$ as a subring of the polynomial ring $A[X]$. Let $\mathcal{N} = \bigoplus_{n \geqslant 0} N_n$ be a graded
 $\mathscr{R}(I)$-module. Let $u\in\mathscr{R}(I)$ be a homogeneous element of degree $s$. Consider the $A$-linear maps given by
 multiplication with $u$:
 \[ N_n \stackrel{u}{\longrightarrow} N_{n+s} \quad \mbox{for all }n. \]
 By applying $\Hom_A(\mathbb{F},-)$ on the above maps and using the naturality of the operators $t_j$, we have the following
 commutative diagram of complexes:
 \[
  \xymatrixrowsep{6mm} \xymatrixcolsep{12mm}
   \xymatrix{
   \Hom_A(\mathbb{F},N_n) \ar[d]^{u} \ar[r]^{t_j} &\Hom_A(\mathbb{F}(+2),N_n) \ar[d]^{u} \\
   \Hom_A(\mathbb{F},N_{n+s}) \ar[r]^{t_j} &\Hom_A(\mathbb{F}(+2),N_{n+s}).
   }
 \]
 Now, taking cohomology, we obtain the following commutative diagram of $A$-modules:
 \[
  \xymatrixrowsep{6mm} \xymatrixcolsep{12mm}
   \xymatrix{
   \Ext_A^i(M,N_n) \ar[d]^{u} \ar[r]^{t_j} & \Ext_A^{i+2}(M,N_n) \ar[d]^{u} \\
   \Ext_A^i(M,N_{n+s}) \ar[r]^{t_j} & \Ext_A^{i+2}(M,N_{n+s})
   }
 \]
 for all $i, n$ and $1 \leqslant j \leqslant c$. Thus
 \[
   \mathscr{E}(\mathcal{N}) := \bigoplus_{i\geqslant 0}\bigoplus_{n\geqslant 0}\Ext_A^i(M,N_n)
 \]
 turns into a bigraded $\mathscr{S} := \mathscr{R}(I)[t_1,\ldots,t_c]$-module, where we set $\deg(t_j) = (0,2)$ for all
 $1 \leqslant j \leqslant c$ and $\deg(u X^s) = (s,0)$ for all $u \in I^s$, $s \geqslant 0$.
 \end{para}
 
 \begin{para}\label{para:module structure 3}
 Suppose $N$ is a finitely generated $A$-module. Set $\mathcal{L} := \bigoplus_{n \geqslant 0}(N/I^{n+1}N)$. Note that
 $\mathscr{R}(I,N) = \bigoplus_{n \geqslant 0} I^n N$ and $N[X] = N \otimes_A A[X]$ are graded modules over $\mathscr{R}(I)$
 and $A[X]$ respectively. Since $\mathscr{R}(I)$ is a graded subring of $A[X]$, we set that $N[X]$ is a graded
 $\mathscr{R}(I)$-module. Therefore $\mathcal{L}$ is a graded $\mathscr{R}(I)$-module, where the graded structure is
 induced by the sequence
 \[ 0 \longrightarrow \mathscr{R}(I,N) \longrightarrow N[X] \longrightarrow \mathcal{L}(-1) \longrightarrow 0. \]
 Therefore, by the observations made in Section~\ref{para:module structure 2}, 
 \[ \mathscr{E}(\mathcal{L}) = \bigoplus_{i \geqslant 0} \bigoplus_{n \geqslant 0} \Ext_A^i(M,N/I^{n+1}N) \]
 is a bigraded module over $\mathscr{S}=\mathscr{R}(I)[t_1,\ldots,t_c]$.
 \end{para}
 
 Let us recall the following result from \cite[1.1]{Pu13}.
 
\begin{theorem}\label{theorem:finitely generated}
 Let $Q$ be a Noetherian ring of finite Krull dimension, and let ${\bf f} = f_1,\ldots,f_c$ be a $Q$-regular sequence. Set $A :=
 Q/({\bf f})$. Let $M$ be a finitely generated $A$-module with $\projdim_Q(M)$ finite. Let $I$ be an ideal of $A$, and let
 $\mathcal{N} = \bigoplus_{n \geqslant 0} N_n$ be a finitely generated $\mathscr{R}(I)$-module. Then 
 \[ \mathscr{E}(\mathcal{N}) := \bigoplus_{i \geqslant 0} \bigoplus_{n \geqslant 0} \Ext_A^i(M,N_n) \]
 is a finitely generated bigraded $\mathscr{S}=\mathscr{R}(I)[t_1,\ldots,t_c]$-module.
\end{theorem}

We first prove the main results of this article for a ring $A$ which is of the form $Q/({\bf f})$, where $Q$ is a regular local ring
and ${\bf f} = f_1,\ldots,f_c$ is a $Q$-regular sequence. Then we deduce the main results for a local complete intersection ring with
the help of the following well-known lemma:

\begin{lemma}\label{lemma:associated}
 Let $(A,\mathfrak{m})$ be a Noetherian local ring and $\widehat{A}$ the completion of $A$ with respect to
 $\mathfrak{m}$. Let $D$ be an $A$-module. Then
 \[\Ass_A(D) = \{\mathfrak{q}\cap A : \mathfrak{q} \in \Ass_{\widehat{A}}(D \otimes_A \widehat{A})\}.\]
\end{lemma}

\section{Asymptotic associated primes: Finiteness}\label{Asymptotic associated primes: Finiteness}

 In this section, we prove the announced finiteness result for the set of associated prime ideals of the family of Ext-modules
 $\Ext_A^i(M, N/I^n N)$, $(i, n \geqslant 0)$, where $M$ and $N$ are finitely generated modules over a local complete intersection ring $A$
 and $I \subseteq A$ is an ideal (see Corollary~\ref{corollary:finiteness}).
 
\begin{theorem}\label{theorem:Q mod f finiteness}
 Let $Q$ be a Noetherian ring of finite Krull dimension, and let ${\bf f} = f_1,\ldots,f_c$ be a $Q$-regular sequence. Set $A :=
 Q/({\bf f})$. Let $M$ and $N$ be two finitely generated $A$-modules with $\projdim_Q(M)$ finite, and let $I$ be an ideal of $A$. Then
 \[ \bigcup_{i \geqslant 0} \bigcup_{n \geqslant 0} \Ass_A \left( \Ext_A^i(M,N/I^n N) \right) \quad \mbox{is a finite set}. \]
\end{theorem}

\begin{proof}
 For a fixed $n \geqslant 0$, consider the short exact sequence of $A$-modules:
 \[ 
   0 \longrightarrow I^nN/I^{n+1}N \longrightarrow N/I^{n+1}N \longrightarrow N/I^nN \longrightarrow 0.
 \]
 Taking direct sum over $n \geqslant 0$ and setting $\mathcal{L} := \bigoplus_{n \geqslant 0}(N/I^{n+1}N)$, we obtain the following short exact
 sequence of graded $\mathscr{R}(I)$-modules:
 \[
   0 \longrightarrow \gr_I(N) \longrightarrow \mathcal{L} \longrightarrow \mathcal{L}(-1) \longrightarrow 0,
 \]
 which induces an exact sequence of graded $\mathscr{R}(I)$-modules for each $i \geqslant 0$:
 \[ 
   \Ext_A^i(M,\gr_I(N)) \longrightarrow \Ext_A^i(M,\mathcal{L}) \longrightarrow \Ext_A^i(M,\mathcal{L}(-1)).
 \]
 Taking direct sum over $i \geqslant 0$ and using the naturality of the cohomology operators $t_j$, we get the following exact sequence of
 bigraded $\mathscr{S} = \mathscr{R}(I)[t_1,\ldots,t_c]$-modules:
 \[
  \bigoplus_{i,n\geqslant 0} \Ext_A^i\left(M,\dfrac{I^nN}{I^{n+1}N}\right) \stackrel{\Phi}{\longrightarrow}
  \bigoplus_{i,n\geqslant 0} V_{i,n} \stackrel{\Psi}{\longrightarrow}
  \bigoplus_{i,n\geqslant 0} V_{i,n-1},
 \]
 where $V_{i,n} := \Ext_A^i(M,N/I^{n+1}N)$ for each $i \geqslant 0$ and $n \geqslant -1$. Let
 \[
   U = \bigoplus_{i,n \geqslant 0} U_{i,n} := \Image(\Phi).
 \]
 Then, for each $i, n \geqslant 0$, considering the exact sequence of $A$-modules:
 \[ 0 \rightarrow U_{i,n} \rightarrow V_{i,n} \rightarrow V_{i,n-1},\]
 we have 
 \begin{align*}
  \Ass_A(V_{i,n})&\subseteq \Ass_A(U_{i,n}) \cup \Ass_A(V_{i,n-1}) \\
                 &\subseteq \Ass_A(U_{i,n}) \cup \Ass_A(U_{i,n-1}) \cup \Ass_A(V_{i,n-2}) \\
                 & ~~\vdots \\
    &\subseteq \bigcup_{0\leqslant j\leqslant n}\Ass_A(U_{i,j}) \quad \mbox{[as $\Ass_A(V_{i,-1}) = \phi$ for each $i \geqslant 0$]}.
 \end{align*}
 Taking union over $i,n \geqslant 0$, we obtain
 \begin{equation}\label{VU containment}
  \bigcup_{i,n\geqslant 0}\Ass_A(V_{i,n}) \subseteq \bigcup_{i,n\geqslant 0}\Ass_A(U_{i,n}).
 \end{equation}
 
 Since $\gr_I(N)$ is a finitely generated graded $\mathscr{R}(I)$-module, by Theorem~\ref{theorem:finitely generated},
 \[ 
  \bigoplus_{i,n\geqslant 0} \Ext_A^i\left(M, \dfrac{I^n N}{I^{n+1} N}\right)
 \]
 is a finitely generated bigraded $\mathscr{S}$-module, and hence $U$ is a finitely generated bigraded $\mathscr{S}$-module.
 Therefore, by \cite[Lemma~3.2]{We04},
 \begin{equation}\label{U finite}
  \bigcup_{i,n\geqslant 0}\Ass_A(U_{i,n})\quad\mbox{ is a finite set.}
 \end{equation}
 Now the result follows from \eqref{VU containment} and \eqref{U finite}.
\end{proof}

An immediate corollary is the following:

\begin{corollary}\label{corollary:finiteness}
 Let $(A,\mathfrak{m})$ be a local complete intersection ring. Let $M,N$ be two finitely generated $A$-modules and
 $I$ an ideal of $A$. Then
 \[\bigcup_{i\geqslant 0}\bigcup_{n\geqslant 0}\Ass_A\left(\Ext_A^i(M,N/I^n N)\right) \quad\mbox{is a finite set.}\]
\end{corollary}

\begin{proof}
 Since $A$ is a local complete intersection ring, $\widehat{A} = Q/({\bf f})$, where $Q$ is a regular
 local ring and ${\bf f}=f_1,\ldots,f_c$ is a $Q$-regular sequence.
 
 Applying Theorem~\ref{theorem:Q mod f finiteness} for the ring $\widehat{A}$, we have that
 \[
   \bigcup_{i,n \geqslant 0} \Ass_{\widehat{A}}\left(\Ext_A^i(M,N/I^nN)\otimes_A\widehat{A}\right) = 
   \bigcup_{i,n \geqslant 0} \Ass_{\widehat{A}}
   \left(\Ext_{\widehat{A}}^i\left(\widehat{M},\widehat{N}/(I\widehat{A})^n\widehat{N}\right)\right)
 \]
 is a finite set, and hence the result follows from Lemma~\ref{lemma:associated}. 
\end{proof}

\section{Asymptotic associated primes: Stability}\label{Asymptotic associated primes: Stability}
 
 In the present section, we analyze the asymptotic behaviour of the sets of associated prime ideals of Ext-modules
 $\Ext_A^i(M, N/I^n N)$, $(i, n \geqslant 0)$, where $M$ and $N$ are finitely generated modules over a local complete intersection ring $A$
 and $I \subseteq A$ is an ideal (see Corollary~\ref{corollary:stability}). We first prove the following theorem:
 
\begin{theorem}\label{theorem:Q mod f stability}
 Let $Q$ be a Noetherian ring of finite Krull dimension, and let ${\bf f} = f_1,\ldots,f_c$ be a $Q$-regular sequence. Set $A :=
 Q/({\bf f})$. Let $M$ and $N$ be two finitely generated $A$-modules with $\projdim_Q(M)$ finite, and let $I$ be an ideal of $A$.
 Then there exist $i_0, n_0 \geqslant 0$ such that for all $i\geqslant i_0$ and $n \geqslant n_0$, we have
 \begin{align*}
  \Ass_A\left(\Ext_A^{2i}(M,N/I^nN)\right) &= \Ass_A\left(\Ext_A^{2 i_0}(M,N/I^{n_0}N)\right), \\
  \Ass_A\left(\Ext_A^{2i+1}(M,N/I^nN)\right) &= \Ass_A\left(\Ext_A^{2 i_0 + 1}(M,N/I^{n_0}N)\right).
 \end{align*}
\end{theorem}

 To prove this result, we assume the following lemma which we prove at the end of this section.
 
\begin{lemma}\label{lemma:polynomial}
 Let $(Q,\mathfrak{n})$ be a Noetherian local ring with residue field $k$, and let ${\bf f} = f_1,\ldots,f_c$ be a $Q$-regular
 sequence. Set $A := Q/({\bf f})$. Let $M$ and $N$ be two finitely generated $A$-modules with $\projdim_Q(M)$ finite, and let $I$ be
 an ideal of $A$. Then
 \[
   \lambda_A\left(\Hom_A\left(k,\Ext_A^{2i}(M,N/I^nN)\right)\right) ~\mbox{ and }~
   \lambda_A\left(\Hom_A\left(k,\Ext_A^{2i+1}(M,N/I^nN)\right)\right)
 \]
 are given by polynomials in $i, n$ with rational coefficients for all sufficiently large $(i,n)$.
\end{lemma}

\begin{proof}[Proof of Theorem~\ref{theorem:Q mod f stability}]
 Applying Theorem~\ref{theorem:Q mod f finiteness}, we may assume that
 \[
 \bigcup_{i\geqslant 0}\bigcup_{n\geqslant 0}\Ass_A\left(\Ext_A^i(M,N/I^n N)\right)
  = \{\mathfrak{p}_1,\mathfrak{p}_2,\ldots,\mathfrak{p}_l\}.
 \]
 Set $V_{i,n} := \Ext_A^i(M,N/I^n N)$ for each $i, n \geqslant 0$, and $V := \bigoplus_{i, n \geqslant 0} V_{i,n}$. 
 
 We first prove that there exist some $i', n' \geqslant 0$ such that
 \begin{equation}\label{equation:even stability}
  \Ass_A(V_{2i,n}) = \Ass_A(V_{2i',n'}) \quad \mbox{for all }i \geqslant i' \mbox{ and } n \geqslant n'.
 \end{equation}
 To prove the claim \eqref{equation:even stability}, it is enough to prove that for each $\mathfrak{p}_j$, where $1 \leqslant j \leqslant l$,
 there exist some $i_{j_0}, n_{j_0} \geqslant 0$ such that exactly one of the following alternatives must hold:
 \begin{align*}
  \mbox{either}\quad \mathfrak{p}_j &\in \Ass_A(V_{2i,n})\quad\mbox{for all }i \geqslant i_{j_0} \mbox{ and } n \geqslant n_{j_0}; \\
  \mbox{or}\quad     \mathfrak{p}_j &\notin \Ass_A(V_{2i,n})\quad\mbox{for all }i \geqslant i_{j_0} \mbox{ and } n \geqslant n_{j_0}.
 \end{align*}
 Localizing at $\mathfrak{p}_j$, and replacing $A_{\mathfrak{p}_j}$ by $A$ and $\mathfrak{p}_j A_{\mathfrak{p}_j}$ by $\mathfrak{m}$,
 it is now enough to prove that there exist some $i', n' \geqslant 0$ such that
 \begin{align}
  \mbox{either}\quad \mathfrak{m} &\in \Ass_A(V_{2i,n})\quad\mbox{for all }i \geqslant i' \mbox{ and } n \geqslant n';   
                                                                                       \label{equation:m in} \\
  \mbox{or}\quad     \mathfrak{m} &\notin \Ass_A(V_{2i,n})\quad\mbox{for all }i \geqslant i' \mbox{ and } n \geqslant n'.
											  \label{equation:m not in}
 \end{align}
 But from Lemma~\ref{lemma:polynomial}, $\lambda_A(\Hom_A(k,V_{2i,n}))$ is given by a polynomial in $i, n$ with rational coefficients
 for all sufficiently large $(i,n)$, and hence there exist $i',n' \geqslant 0$ such that
 \begin{align*}
  \mbox{either}\quad \Hom_A(k,V_{2i,n}) &\neq 0 \quad\mbox{for all }i \geqslant i' \mbox{ and } n \geqslant n'; \\
  \mbox{or}\quad     \Hom_A(k,V_{2i,n}) &= 0\quad\mbox{for all }i \geqslant i' \mbox{ and } n \geqslant n',
 \end{align*}
 which is equivalent to that either \eqref{equation:m in} is true, or \eqref{equation:m not in} is true.
 
 Applying a similar procedure as in the even case, we see that there exist $i'', n'' \geqslant 0$ such that
 \begin{equation*}
  \Ass_A(V_{2i+1,n}) = \Ass_A(V_{2i''+1,n''})\quad\mbox{for all }i \geqslant i'' \mbox{ and } n \geqslant n''.
 \end{equation*}
 Now $(i_0,n_0) := \max\{(i',n'),(i'',n'')\}$ satisfies the required result of the theorem.
\end{proof}

An immediate corollary of the Theorem~\ref{theorem:Q mod f stability} is the following:

\begin{corollary}\label{corollary:stability}
 Let $(A,\mathfrak{m})$ be a local complete intersection ring. Let $M$ and $N$ be two finitely generated $A$-modules, and let $I$ be
 an ideal of $A$. Then there exist $i_0, n_0 \geqslant 0$ such that for all $i \geqslant i_0$ and $n \geqslant n_0$, we have
 \begin{align*}
  \Ass_A\left(\Ext_A^{2i}(M,N/I^nN)\right) &= \Ass_A\left(\Ext_A^{2 i_0}(M,N/I^{n_0}N)\right), \\
  \Ass_A\left(\Ext_A^{2i+1}(M,N/I^nN)\right) &= \Ass_A\left(\Ext_A^{2 i_0 + 1}(M,N/I^{n_0}N)\right).
 \end{align*}
\end{corollary}

\begin{proof}
 Assume $\widehat{A} = Q/({\bf f})$, where $Q$ is a regular local ring and ${\bf f}=f_1,\ldots,f_c$ is a $Q$-regular
 sequence. Then, by applying Theorem~\ref{theorem:Q mod f stability} for the ring $\widehat{A}$, we see that there exist $i_0, n_0
 \geqslant 0$ such that for all $i \geqslant i_0$ and $n \geqslant n_0$, we have
 \begin{align*}
  \Ass_{\widehat{A}}\left(\Ext_A^{2i}(M,N/I^nN)\otimes_A\widehat{A}\right) & = 
  \Ass_{\widehat{A}}\left(\Ext_A^{2i_0}(M,N/I^{n_0}N)\otimes_A\widehat{A}\right), \\
  \Ass_{\widehat{A}}\left(\Ext_A^{2i+1}(M,N/I^nN)\otimes_A\widehat{A}\right) & = 
  \Ass_{\widehat{A}}\left(\Ext_A^{2i_0+1}(M,N/I^{n_0}N)\otimes_A\widehat{A}\right).
 \end{align*}
 The result now follows from Lemma~\ref{lemma:associated}.
\end{proof}

We now give

\begin{proof}[Proof of Lemma~\ref{lemma:polynomial}]
 For each fixed $n \geqslant 0$, consider the short exact sequence of $A$-modules:
 \[ 
  0 \longrightarrow I^n N \longrightarrow N \longrightarrow N/I^n N \longrightarrow 0,
 \]
 which induces an exact sequence of $A$-modules for each $i, n$:
 \[
  \Ext_A^i(M,I^n N) \rightarrow \Ext_A^i(M,N) \rightarrow \Ext_A^i(M,N/I^n N) \rightarrow \Ext_A^{i+1}(M,I^n N).
 \]
 Taking direct sum over $i, n$ and using the naturality of the cohomology operators $t_j$, we obtain an exact sequence of bigraded
 $\mathscr{S} = \mathscr{R}(I)[t_1,\ldots,t_c]$-modules:
 \[
  U \longrightarrow T \longrightarrow V \longrightarrow U(0,1),
 \]
 where
 \begin{align*}
  U &= \bigoplus_{i,n\geqslant 0}U_{i,n} := \bigoplus_{i,n\geqslant 0}\Ext_A^i(M,I^n N), \\
  T &= \bigoplus_{i,n\geqslant 0}T_{i,n} := \bigoplus_{i,n\geqslant 0}\Ext_A^i(M,N), \\
  V &= \bigoplus_{i,n\geqslant 0}V_{i,n} := \bigoplus_{i,n\geqslant 0}\Ext_A^i(M,N/I^n N), \quad\mbox{and}
 \end{align*}
 $U(0,1)$ is same as $U$ but the grading is twisted by $(0,1)$. Setting
 \[
  X := \Image(U \to T), ~ Y := \Image(T \to V) ~\mbox{ and }~ Z := \Image(V \to U(0,1)),
 \]
 we have the following commutative diagram of exact sequences of bigraded $\mathscr{S}$-modules:
 \[
  \xymatrixrowsep{2mm} \xymatrixcolsep{5mm}
  \xymatrix{
	    U \ar[rd] \ar[rr] && T \ar[rd] \ar[rr] && V \ar[rd] \ar[rr] && U(0,1) \\
	    & X \ar[ru] \ar[rd] && Y \ar[ru] \ar[rd] && Z \ar[ru] \ar[rd]     \\
	    0 \ar[ru] && 0 \ar[ru] && 0 \ar[ru] && 0, }
 \]
 which gives the following two short exact sequences of bigraded $\mathscr{S}$-modules:
 \[
  0 \rightarrow X \rightarrow T \rightarrow Y \rightarrow 0 \quad \mbox{and} \quad
  0 \rightarrow Y \rightarrow V \rightarrow Z \rightarrow 0.
 \]
 Now applying $\Hom_A(k,-)$ to these short exact sequences, we get the following exact sequences of bigraded $\mathscr{S}$-modules:
 \begin{align}
  &0 \longrightarrow \Hom_A(k,X) \longrightarrow \Hom_A(k,T) \longrightarrow \Hom_A(k,Y) \longrightarrow C \longrightarrow 0,
													    \label{equation:es 1} \\
  &0 \longrightarrow \Hom_A(k,Y) \longrightarrow \Hom_A(k,V) \longrightarrow D \longrightarrow 0,        \label{equation:es 2}
 \end{align}
 where $C := \Image\left( \Hom_A(k,Y) \longrightarrow \Ext_A^1(k,X) \right)$, \\
 \hspace*{1.03cm} $D := \Image\left( \Hom_A(k,V) \longrightarrow \Hom_A(k,Z) \right)$.
 
 By virtue of Theorem~\ref{theorem:finitely generated}, $U = \bigoplus_{i,n\geqslant 0}\Ext_A^i(M,I^n N)$ is a finitely generated bigraded
 $\mathscr{S}$-module, and hence $X = \Image(U \rightarrow T)$ is so. Therefore $\Hom_A(k,X)$ and $\Ext_A^1(k,X)$ are finitely
 generated bigraded $\mathscr{S}$-modules. Being an $\mathscr{S}$-submodule of $\Ext_A^1(k,X)$, $C$ is
 also a finitely generated  bigraded $\mathscr{S}$-module. Since $\Hom_A(k,X)$ and $\Ext_A^1(k,X)$ are annihilated by
 the maximal ideal of $A$, $\Hom_A(k,X_{i,n})$ and $C_{i,n}$ both are finitely generated 
 $k$-modules, and hence they have finite length as $A$-modules for each $i,n \geqslant 0$. Therefore, by applying the Hilbert-Serre
 Theorem to the bigraded $\mathscr{S}$-modules $\Hom_A(k,X)$ and $C$, we obtain
 \begin{align}
  &\sum_{i,n \geqslant 0} \lambda_A(\Hom_A(k,X_{i,n})) z^i w^n = \dfrac{P_X(z,w)}{{(1-z^2)}^c {(1-w)}^r},\label{equation:series X} \\
  &\sum_{i,n\geqslant 0}\lambda_A(C_{i,n}) z^i w^n = \dfrac{P_C(z,w)}{{(1-z^2)}^c {(1-w)}^r}             \label{equation:series C}
 \end{align}
 for some polynomials $P_X(z,w)$ and $P_C(z,w)$ over $\mathbb{Z}$.
 
 For a fixed $n \geqslant 0$, $\bigoplus_{i \geqslant 0} T_{i,n} = \bigoplus_{i \geqslant 0} \Ext_A^i(M,N)$
 is a finitely generated graded $A[t_1,\ldots,t_c]$-module, and hence $\Hom_A\left(k,\bigoplus_{i\geqslant 0}T_{i,n}\right)$ is also so.
 By similar argument as before, $\lambda_A(\Hom_A(k,T_{i,n}))<\infty$ for each $i \geqslant 0$. Therefore, again by the 
 Hilbert-Serre Theorem, we have for each $n \geqslant 0$:
 \begin{align*}
  \sum_{i\geqslant 0}\lambda_A(\Hom_A(k,T_{i,n})) z^i = \dfrac{P_S(z)}{{(1-z^2)}^c}
 \end{align*}
 for some polynomial $P_S(z)\in \mathbb{Z}[z]$. Multiplying both sides of the above equation by $w^n$, and then taking
 sums over $n \geqslant 0$, we get
 \begin{equation}\label{equation:series T}
  \sum_{i,n \geqslant 0}\lambda_A(\Hom_A(k,T_{i,n}))z^i w^n = \dfrac{P_S(z)}{{(1-z^2)}^c (1-w)} = \dfrac{P_T(z,w)}{{(1-z^2)}^c {(1-w)}^r},
 \end{equation}
 where $P_T(z,w) := P_S(z) {(1-w)}^{r-1} \in \mathbb{Z}[z,w]$.
 
 Now considering \eqref{equation:es 1}, we have an exact sequence of $A$-modules:
 \[
  0 \longrightarrow \Hom_A(k,X_{i,n}) \longrightarrow \Hom_A(k,T_{i,n}) \longrightarrow \Hom_A(k,Y_{i,n}) \longrightarrow C_{i,n}
  \longrightarrow 0
 \]
 for each $i, n \geqslant 0$. So, the additivity of length functions gives
 \[
  \lambda_A(\Hom_A(k,X_{i,n})) - \lambda_A(\Hom_A(k,T_{i,n})) + \lambda_A(\Hom_A(k,Y_{i,n})) - \lambda_A(C_{i,n})=0
 \]
 for each $i,n\geqslant 0$. Multiplying the above equation by $z^i w^n$, then taking sum over $i, n \geqslant 0$, and using
 \eqref{equation:series X}, \eqref{equation:series C} and \eqref{equation:series T}, we obtain
 \begin{equation}\label{equation:series Y}
  \sum_{i,n\geqslant 0}\lambda_A(\Hom_A(k,Y_{i,n}))z^i w^n = \dfrac{P_Y(z,w)}{{(1-z^2)}^c {(1-w)}^r},
 \end{equation}
 where $P_Y(z,w) := - P_X(z,w) + P_T(z,w) + P_C(z,w) \in \mathbb{Z}[z,w]$.
 
 Recall that by Theorem~\ref{theorem:finitely generated}, $U$ is a finitely generated bigraded $\mathscr{S}$-module. Now observe that
 $Z$ is a bigraded submodule of $U(0,1)$. Therefore $Z$ is a finitely generated bigraded $\mathscr{S}$-module, hence $\Hom_A(k,Z)$ is
 so, and hence
 \[
  D = \Image(\Hom_A(k,V) \longrightarrow \Hom_A(k,Z))
 \]
 is also so. Observe that $\lambda_A(D_{i,n})<\infty$ for each $i,n\geqslant 0$.
 Therefore, once again by applying the Hilbert-Serre Theorem, we have
 \begin{equation}\label{equation:series D}
  \sum_{i,n \geqslant 0} \lambda_A(D_{i,n})z^i w^n = \dfrac{P_D(z,w)}{{(1-z^2)}^c {(1-w)}^r}
 \end{equation}
 for some polynomial $P_D(z,w)$ over $\mathbb{Z}$.
 
 Now considering \eqref{equation:es 2}, we have an exact sequence of $A$-modules:
 \[
  0 \longrightarrow \Hom_A(k,Y_{i,n}) \longrightarrow \Hom_A(k,V_{i,n}) \longrightarrow D_{i,n} \longrightarrow 0
 \]
 for each $i,n\geqslant 0$, which gives
 \[\lambda_A(\Hom_A(k,Y_{i,n})) - \lambda_A(\Hom_A(k,V_{i,n})) + \lambda_A(D_{i,n}) = 0.\]
 Multiplying the above equation by $z^i w^n$, then taking sum over $i, n \geqslant 0$, and using
 \eqref{equation:series Y} and \eqref{equation:series D}, we obtain
 \begin{equation*}
  \sum_{i, n \geqslant 0}\lambda_A(\Hom_A(k,V_{i,n}))z^i w^n = \dfrac{P_V(z,w)}{{(1-z^2)}^c {(1-w)}^r},
 \end{equation*}
 where $P_V(z,w) := P_Y(z,w) + P_D(z,w) \in \mathbb{Z}[z,w]$. Therefore it follows that
 \[\lambda_A(\Hom_A(k,V_{2i,n}))\quad\mbox{ and }\quad \lambda_A(\Hom_A(k,V_{2i+1,n}))\]
 are given by polynomials in $i,n$ with rational coefficients for all sufficiently large $(i,n)$, which completes the
 proof of the lemma.
\end{proof}

\section{Asymptotic associated primes: The geometric case}\label{Asymptotic associated primes: The geometric case}

Let $V$ be an affine or projective variety over an algebraically closed field $K$. Let $A$ be the coordinate ring of $V$.
Then $V$ is said to be a {\it locally complete intersection variety} if all its local rings are complete intersection. Thus
\begin{itemize}
 \item in the affine case, $A_\mathfrak{p}$ is a local complete intersection ring for every $\mathfrak{p}\in \Spec(A)$.
 \item in the projective case, $A_{(\mathfrak{p})}$ is a local complete intersection ring for every
 $\mathfrak{p}\in \Proj(A)$. Recall that $A_{(\mathfrak{p})}$ is the degree zero part of the graded ring
 $S^{-1}A$, where $S$ is the collection of all homogeneous elements in $A\smallsetminus \mathfrak{p}$.
\end{itemize}

In this section, we prove the results analogous to Theorems \ref{theorem:Q mod f finiteness} and
\ref{theorem:Q mod f stability} for the coordinate rings of locally complete intersection varieties.
In the affine case, we prove the following general result:

\begin{theorem}\label{theorem:affine case}
 Let $A = Q/\mathfrak{a}$, where $Q$ is a regular ring of finite Krull dimension and $\mathfrak{a} \subseteq Q$ is an ideal so that
 $\mathfrak{a}_{\mathfrak{q}} \subseteq Q_{\mathfrak{q}}$ is generated by a $Q_{\mathfrak{q}}$-regular sequence for each
 $\mathfrak{q} \in \Var(\mathfrak{a})$. Let $M$ and $N$ be two finitely generated $A$-modules, and let $I$ be an ideal of $A$.
 Then the set
 \[
  \bigcup_{i \geqslant 0} \bigcup_{n \geqslant 0} \Ass_A\left( \Ext_A^i(M, N/I^n N) \right) \quad \mbox{is finite}.
 \]
 Moreover, there exist $i_0,n_0 \geqslant 0$ such that for all $i \geqslant i_0$ and $n \geqslant n_0$, we have
 \begin{align*}
  \Ass_A\left(\Ext_A^{2i}(M,N/I^nN)\right) &= \Ass_A\left(\Ext_A^{2 i_0}(M,N/I^{n_0}N)\right), \\
  \Ass_A\left(\Ext_A^{2i+1}(M,N/I^nN)\right) &= \Ass_A\left(\Ext_A^{2 i_0 + 1}(M,N/I^{n_0}N)\right).
 \end{align*}
\end{theorem}

\begin{proof}
 For $x\in A$, set $D(x)=\{\mathfrak{p}\in \Spec(A) : x\notin \mathfrak{p}\}$. As in
 \cite[{\it Proof of Theorem~6.1}]{Pu13}, we have
 \[
  \Spec(A) = D(g_1) \cup \cdots \cup D(g_m) \quad \mbox{for some } g_1,\ldots,g_m \in A
 \]
 such that the localization $A_{g_j}$ by $\{{g_j}^l : l \geqslant 0\}$ has the form $Q_j/\mathfrak{a}_j$ for some regular ring
 $Q_j$ of finite Krull dimension and some ideal $\mathfrak{a}_j$ of $Q_j$ generated by a $Q_j$-regular sequence.
 Clearly, for any $A$-module $E$, we obtain
 \begin{equation}\label{equation:affine 1}
  \Ass_A(E) = \bigcup\{\mathfrak{q}\cap A : \mathfrak{q} \in \Ass_{A_{g_j}}(E_{g_j})\mbox{ for some }j = 1,\ldots,m\}.
 \end{equation}
 
 Since localization $A_{g_j}$ is flat over $A$, we have
 \begin{equation}\label{equation:affine 2}
 \left(\Ext_A^i(M,N/I^n N)\right)_{g_j} = ~\Ext_{A_{g_j}}^i\left(M_{g_j},N_{g_j}/(IA_{g_j})^n N_{g_j}\right)
 \end{equation}
 for all $i,n\geqslant 0$ and $j=1,\ldots,m$. Therefore, from \eqref{equation:affine 1} and \eqref{equation:affine 2}, it is
 enough to prove the result for the ring $A_{g_j}=Q_j/\mathfrak{a}_j$ for each $j$. Note that $Q_j$ is a regular ring
 of finite Krull dimension, and hence $\projdim_{Q_j}(M_{g_j})$ is finite. Therefore the result now follows by
 applying the Theorems \ref{theorem:Q mod f finiteness} and \ref{theorem:Q mod f stability} to each
 $A_{g_j}=Q_j/\mathfrak{a}_j$.
\end{proof}
 
 Now we prove the analogous result to Theorem~\ref{theorem:affine case} in the projective case. Let us fix the following hypothesis:
 
 \begin{hypothesis}\label{hypothesis: the projective locally CI variety}
  Let $K$ be a field not necessarily algebraically closed, and let $Q = K[X_0,X_1,\ldots,X_r]$ be a polynomial ring over $K$,
  where $\deg(X_i) = 1$ for all $i$. Let $\mathfrak{a}$ be a homogeneous ideal of $Q$. Set $A := Q/\mathfrak{a}$. Suppose
  $A_{(\mathfrak{p})}$ is a complete intersection ring for every $\mathfrak{p} \in \Proj(A)$.
  
  Let $\mathfrak{m}$ be the unique maximal homogeneous ideal of $A$. If $E$ is a graded $A$-module, then define the set of relevant
  associated primes of $E$ as
  \[
    {}^*\Ass_A(E) := \Ass_A(E) \smallsetminus \{\mathfrak{m}\}.
  \]
 \end{hypothesis}
 
 \begin{theorem}\label{theorem:projective case}
  With the {\rm Hypothesis~\ref{hypothesis: the projective locally CI variety}}, let $M$ and $N$ be two finitely generated graded
  $A$-modules, and let $I$ be a homogeneous ideal of $A$. Then the set
  \[
    \bigcup_{i \geqslant 0} \bigcup_{n \geqslant 0} {}^*\Ass_A \left( \Ext_A^i(M, N/I^n N) \right) \quad \mbox{is finite}.
  \]
  Moreover, there exist $i_0,n_0 \geqslant 0$ such that for all $i \geqslant i_0$ and $n \geqslant n_0$, we have
  \begin{align*}
   {}^*\Ass_A\left(\Ext_A^{2i}(M,N/I^nN)\right) &= {}^*\Ass_A\left(\Ext_A^{2 i_0}(M,N/I^{n_0}N)\right), \\
   {}^*\Ass_A\left(\Ext_A^{2i+1}(M,N/I^nN)\right) &= {}^*\Ass_A\left(\Ext_A^{2 i_0 + 1}(M,N/I^{n_0}N)\right).
  \end{align*}
 \end{theorem}
 
 \begin{proof}
  For a homogeneous element $x\in A$, set ${}^*D(x) = \{\mathfrak{p} \in \Proj(A) : x \notin \mathfrak{p}\}$. As in
 \cite[{\it Proof of Theorem~6.3}]{Pu13}, we have
 \[
  \Proj(A) = {}^*D(g_1) \cup \cdots \cup {}^*D(g_m) \quad \mbox{for some homogeneous } g_1,\ldots,g_m \in A
 \]
 such that $A_{g_j} = Q_j/\mathfrak{a}_j$ for some regular ring $Q_j$ of finite Krull dimension and some homogeneous
 ideal $\mathfrak{a}_j$ of $Q_j$ generated by a $Q_j$-regular sequence. Clearly, for any graded $A$-module $E$, we obtain
 \[
  {}^*\Ass_A(E) = \bigcup\{\mathfrak{q}\cap A : \mathfrak{q} \in \Ass_{A_{g_j}}(E_{g_j})\mbox{ for some }j=1,\ldots,m\}.
 \]
 Similarly, as in the proof of Theorem~\ref{theorem:affine case}, the result follows by applying the
 Theorems~\ref{theorem:Q mod f finiteness} and \ref{theorem:Q mod f stability} to each $A_{g_j} = Q_j/\mathfrak{a}_j$.
 \end{proof}
 
\section{Support varieties}\label{Support varieties}

Let $(A,\mathfrak{m})$ be a local complete intersection ring of codimension $c$. Let $M$ and $N$ be two finitely
generated $A$-modules. The complexity of the pair $(M,N)$ is defined to be the number
\[
 \cx_A(M,N) = \inf\left\{b\in\mathbb{N} ~\middle|~ \limsup_{n\rightarrow\infty}\dfrac{\mu(\Ext_A^n(M,N))}{n^{b-1}}<\infty \right\},
\]
where $\mu(D)$ denotes the minimal number of generators of a finitely generated $A$-module $D$. 

In this section, we prove the following theorem:

\begin{theorem}\label{theorem:Stability of complexity}
 Let $(A,\mathfrak{m})$ be a local complete intersection ring. Let $M$ and $N$ be two finitely generated $A$-modules, and let
 $I$ be an ideal of $A$. Then
 \[
  \cx_A(M,N/I^jN)\quad\mbox{is constant for all $j\gg 0$.}
 \]
\end{theorem}

As in \cite[Section~7.2]{Pu13}, we may assume that our local complete intersection $A$
\begin{itemize}
 \item[(1)] is complete. So $A=Q/( f_1,\ldots,f_c)$, where $(Q,\mathfrak{n})$ is a regular local ring and
       ${\bf f}=f_1,\ldots,f_c\in\mathfrak{n}^2$ is a $Q$-regular sequence.
 \item[(2)] has an algebraically closed residue field $k$.
\end{itemize}

Let $U$ and $V$ be two finitely generated  $A$-modules. Let 
\[
 \Ext_A^{\star}(U,V) := \bigoplus_{i\geqslant 0}\Ext_A^i(U,V)
\]
be the total Ext module of $U,V$ over the ring $A[t_1,\ldots,t_c]$ of cohomology operators $t_j$ defined by ${\bf f}$,
where $\deg(t_j) = 2$ for all $j = 1,\ldots,c$. Set 
\[
 \mathcal{C}(U,V) := \Ext_A^\star(U,V)\otimes_A k.
\]
Since $\Ext_A^{\star}(U,V)$ is a finitely generated graded $A[t_1,\ldots,t_c]$-module, $\mathcal{C}(U,V)$ is a finitely
generated graded module over 
\[
 \overline{T} := A[t_1,\ldots,t_c] \otimes_A k = k[t_1,\ldots,t_c].
\]

Define the support variety $\mathscr{V}(U,V)$ of $U,V$ as the zero set in $k^c$ of the annihilator of
$\mathcal{C}(U,V)$ in $\overline{T}$, that is
\[
 \mathscr{V}(U,V) := \{(b_1,\ldots,b_c)\in k^c:P(b_1,\ldots,b_c)=0 ~\forall~P\in\ann_{\overline{T}}(\mathcal{C}(U,V))\} \cup \{0\}.
\]
By \cite[2.4(2)]{Av00}, we have 
\begin{center}
 ($@$)\hfill $\cx_A(U,V) = \dim(\mathscr{V}(U,V)) = \dim_{\overline{T}}(\mathcal{C}(U,V))$. \hfill \;
\end{center}

Before proving the Theorem~\ref{theorem:Stability of complexity}, we give the following well-known lemma:

\begin{lemma}\label{lemma:stability of dim}
 Let $\mathscr{R} = \bigoplus_{n \geqslant 0} \mathscr{R}_n$ be a Noetherian standard $\mathbb{N}$-graded ring, and let
 $\mathscr{M} = \bigoplus_{n \geqslant 0} \mathscr{M}_n$ be a finitely generated $\mathbb{N}$-graded $\mathscr{R}$-module. Then there
 exists $j_0 \geqslant 0$ such that
 \[
  \dim_{\mathscr{R}_0}(\mathscr{M}_j) = \dim_{\mathscr{R}_0}(\mathscr{M}_{j_0})\quad\mbox{for all }j\geqslant j_0.
 \]
\end{lemma}

\begin{proof}
 Since $\mathscr{M}$ is a finitely generated $\mathbb{N}$-graded module over a Noetherian standard $\mathbb{N}$-graded
 ring $\mathscr{R}$, there exists $j_0 \geqslant 0$ such that 
 \[ 
 \ann_{\mathscr{R}_0}(\mathscr{M}_j) = \ann_{\mathscr{R}_0}(\mathscr{M}_{j_0})\quad\mbox{for all }j\geqslant j_0,
 \]
 and hence $\dim_{\mathscr{R}_0}(\mathscr{M}_j) = \dim_{\mathscr{R}_0}(\mathscr{M}_{j_0})$ for all $j\geqslant j_0$.
\end{proof}

\begin{para}\label{observation}
 Since $\gr_I(N) = \bigoplus_{j \geqslant 0}I^jN/I^{j+1}N$~ is a finitely generated graded $\mathscr{R}(I)$-module, by
 Theorem~\ref{theorem:finitely generated},
 \[
  \bigoplus_{j \geqslant 0} \Ext_A^{\star}(M,I^jN/I^{j+1}N) = \bigoplus_{i \geqslant 0} \bigoplus_{j \geqslant 0} \Ext_A^i(M,I^jN/I^{j+1}N)
 \]
 is a finitely generated bigraded $\mathscr{R}(I)[t_1,\ldots,t_c]$-module, and hence 
 \[
  \bigoplus_{j \geqslant 0}\Ext_A^\star(M,I^jN/I^{j+1}N) \otimes_A k
 \]
 is a finitely generated graded $\mathscr{R}(I)[t_1,\ldots,t_c] \otimes_A k = F(I)[t_1,\ldots,t_c]$-module, where
 $F(I)$ is the fiber cone of $I$ which is a finitely generated $k$-algebra. Writing
 \[
  F(I)[t_1,\ldots,t_c] = k[x_1,\ldots,x_m][t_1,\ldots,t_c] = \overline{T}[x_1,\ldots,x_m],
 \]
 and using Lemma~\ref{lemma:stability of dim}, we have
 \[ 
  \dim_{\overline{T}}\left(\Ext_A^\star(M,I^jN/I^{j+1}N) \otimes_A k\right) \quad \mbox{is constant for all }j \gg 0,
 \]
 and hence $\cx_A(M,I^jN/I^{j+1}N)$ is constant for all $j \gg 0$, by ($@$).
\end{para}

Now we give

\begin{proof}[Proof of Theorem~\ref{theorem:Stability of complexity}]
 Fix $j \geqslant 0$. Consider the short exact sequence of $A$-modules
 \[
  0 \longrightarrow I^jN/I^{j+1}N \longrightarrow N/I^{j+1}N \longrightarrow N/I^jN \longrightarrow 0,
 \]
 which induces the following exact sequence of $A$-modules for each $i$:
 \begin{align*}
				              \Ext_A^{i-1}(M,N/I^jN) &\longrightarrow \\
  \Ext_A^i(M,I^jN/I^{j+1}N) \longrightarrow \Ext_A^i(M,N/I^{j+1}N) \longrightarrow \Ext_A^i(M,N/I^jN) &\longrightarrow \\
  \Ext_A^{i+1}(M,I^jN/I^{j+1}N). \quad\quad\quad\quad\quad\quad\quad ~~~~~~~~~~~~~~~~~~~~~~~~~~~~~~~~~~~~~~ &
 \end{align*}
 Taking direct sum over $i$ and setting
 \[
   U_j := \bigoplus_{i \geqslant 0}\Ext_A^i(M,I^jN/I^{j+1}N)\quad\mbox{and}\quad V_j := \bigoplus_{i \geqslant 0}\Ext_A^i(M,N/I^jN),
 \]
 we obtain an exact sequence of $A[t_1,\ldots,t_c]$-modules:
 \[ 
   V_j(-1) \stackrel{\varphi_1}{\longrightarrow} U_j \stackrel{\varphi_2}{\longrightarrow}
   V_{j+1} \stackrel{\varphi_3}{\longrightarrow} V_j  \stackrel{\varphi_4}{\longrightarrow} U_j(1). 
 \]
 Set $Z_j := \Image(\varphi_1)$, $X_j := \Image(\varphi_2)$ and $Y_j := \Image(\varphi_3)$.
 Thus we have the following commutative diagram of exact sequences:
 \[
  \xymatrixrowsep{2mm} \xymatrixcolsep{4mm}
  \xymatrix{
	      V_j(-1) \ar[rd] \ar[rr] && U_j \ar[rd] \ar[rr] && V_{j+1} \ar[rd] \ar[rr] && V_j \ar[rd] \ar[rr] && U_j(1) \\
	      & Z_j \ar[ru] \ar[rd] && X_j \ar[ru] \ar[rd] && Y_j \ar[ru] \ar[rd] && Z_j(1) \ar[ru] \ar[rd] \\
	      0 \ar[ru] && 0 \ar[ru] && 0 \ar[ru] && 0 \ar[ru] && 0. }
 \]
 Now consider the short exact sequences of $A[t_1,\ldots,t_c]$-modules:
 \[
  0 \rightarrow X_j \rightarrow V_{j+1} \rightarrow Y_j \rightarrow 0 \quad\mbox{and}\quad
  0 \rightarrow Y_j \rightarrow V_j \rightarrow Z_j(1) \rightarrow 0.
 \]
 Tensoring these sequences with $k$ over $A$, we get the following exact sequences of $\overline{T} = k[t_1,\ldots,t_c]$-modules:
 \begin{align*}
  X_j\otimes_A k \stackrel{\Phi_j}{\longrightarrow} &V_{j+1}\otimes_A k \longrightarrow Y_j\otimes_A k \longrightarrow 0,\\
  Y_j\otimes_A k \stackrel{\Psi_j}{\longrightarrow} &V_j\otimes_A k  \longrightarrow Z_j(1)\otimes_A k \longrightarrow 0.
 \end{align*}
 Now for each $j\geqslant 0$, set $X'_j:=\Image(\Phi_j)$ and $Y'_j:=\Image(\Psi_j)$ to get the following short exact
 sequences of $\overline{T}$-modules:
 \begin{align}
  0 \longrightarrow X'_j \longrightarrow & V_{j+1}\otimes_A k \longrightarrow Y_j\otimes_A k \longrightarrow 0,
											      \label{equation:support 1}\\
  0 \longrightarrow Y'_j \longrightarrow & V_j\otimes_A k  \longrightarrow Z_j(1)\otimes_A k \longrightarrow 0.
											      \label{equation:support 2}
 \end{align}
 
 From Section~\ref{observation}, $\bigoplus_{j \geqslant 0} U_j$ is a finitely generated graded
 $\mathscr{R}(I) [t_1, \ldots, t_c]$-module, and hence its submodule $\bigoplus_{j \geqslant 0} Z_j$ is also so.
 Therefore $\bigoplus_{j \geqslant 0} (Z_j\otimes_A k)$ is a finitely generated graded module over
 \[\mathscr{R}(I)[t_1,\ldots,t_c]\otimes_A k = F(I)[t_1,\ldots,t_c] = \overline{T}[x_1,\ldots,x_m].\]
 Therefore, by Lemma~\ref{lemma:stability of dim}, $\dim_{\overline{T}}(Z_j \otimes_A k) = z$ for all sufficiently large $j$, where $z$
 is some constant. Now considering the short exact sequences \eqref{equation:support 1} and \eqref{equation:support 2}, we obtain that
 \begin{align}
  \dim_{\overline{T}}(V_{j+1}\otimes_A k) &= \max\{ \dim_{\overline{T}}(X'_j), \dim_{\overline{T}}(Y_j\otimes_A k)\},
											  \label{equation:support dim 1} \\
  \dim_{\overline{T}}(V_j\otimes_A k) &= \max\{ \dim_{\overline{T}}(Y'_j), z \} \geqslant z  \label{equation:support dim 2}
 \end{align}
 for all sufficiently large $j$, say $j\geqslant j_0$.
 
 Note that $\dim_{\overline{T}}(V_j \otimes_A k) = \cx_A(M,N/I^jN)$ for all $j \geqslant 0$. Thus it is enough to prove that the stability
 of $\dim_{\overline{T}}(V_j \otimes_A k)$ holds for all sufficiently large $j$.
 
 If $\dim_{\overline{T}}(V_j \otimes_A k) = z$ for all $j \geqslant j_0$, then we are done. Otherwise there exists some $j \geqslant j_0$
 such that $\dim_{\overline{T}}(V_j \otimes_A k) > z$, and hence for this $j$ we have
 \[
  \dim_{\overline{T}}(V_j \otimes_A k) = \dim_{\overline{T}}(Y'_j) \leqslant \dim_{\overline{T}}(Y_j \otimes_A k)
  \leqslant \dim_{\overline{T}}(V_{j+1} \otimes_A k).
 \]
 First equality above occurs from \eqref{equation:support dim 2}, second inequality occurs because $Y'_j$ is a
 quotient module of $Y_j\otimes_A k$, and the last inequality occurs from \eqref{equation:support dim 1}.
 
 Note that $\dim_{\overline{T}}(V_{j+1} \otimes_A k) > z$. So, by applying a similar procedure, we have
 \[
  \dim_{\overline{T}}(V_{j+1} \otimes_A k) \leqslant \dim_{\overline{T}}(V_{j+2} \otimes_A k).
 \]
 In this way, we have a bounded non-decreasing sequence
 \[ 
  \dim_{\overline{T}}(V_j \otimes_A k) \leqslant \dim_{\overline{T}}(V_{j+1} \otimes_A k) \leqslant \dim_{\overline{T}}(V_{j+2} \otimes_A k)
  \leqslant \cdots \leqslant \dim(\overline{T}) < \infty,
 \]
 which eventually stabilize somewhere, and hence the required stability holds.
\end{proof}

\section*{Acknowledgements}
 We thank the anonymous reviewer for his/her careful reading of our manuscript and many valuable comments.
 The first author would like to thank NBHM, DAE, Govt. of India for providing financial support for this study.

 Department of Mathematics\\
 Indian Institute of Technology Bombay\\
 Powai\\
 Mumbai 400076\\
 India\\
 E-mail: dipankar@math.iitb.ac.in\\
 \hspace*{1.35cm} tputhen@math.iitb.ac.in
 
\end{document}